\documentclass[11pt]{amsart}
\usepackage{amsfonts,amsmath,amssymb,amsthm}
\usepackage{color}
\usepackage{cite}

\newtheorem{theorem}{Theorem}
\theoremstyle{plain}

\newtheorem{corollary}{Corollary}

\newtheorem{definition}{Definition}

\newtheorem{lemma}{Lemma}
\newtheorem{notation}{Notation}

\newtheorem{remark}{Remark}

\numberwithin{equation}{section}

\begin{document}
\title{Pseudo-Projective Tensor on Sequential Warped Products}
\author{S\.INEM G\"ULER}
\address[S. G\"uler]{Department of Industrial Engineering,
Istanbul Sabahattin Zaim University, Halkal\.i, Istanbul, Turkey.}
\email{sinem.guler@izu.edu.tr}
\author{B\"ULENT \"UNAL}
\address[B. \"{U}nal]{Department of Mathematics, Bilkent University,
Bilkent, 06800 Ankara, Turkey}
\email{bulentunal@mail.com}

\begin{abstract}
The main objective of this paper is to study pseudo-projective tensor on sequential warped products and
then to obtain necessary and sufficient conditions for a sequential warped product to be pseudo-projectively flat.
Moreover, we also provide characterization of pseudo-projectively flat sequential generalized Robertson-Walker
and pseudo-projectively flat sequential standard static spacetimes.
\end{abstract}

\maketitle

\section{Introduction}

Let $(M,g)$ be an $n$-dimensional $(n>2)$ pseudo-Riemannian manifold. Then the pseudo-projective
curvature tensor ${\rm P}$ on $M$ including the projective curvature tensor is defined by \cite{prasad}
\begin{eqnarray}
\label{pseudo projective curvature tensor}
{\rm P}(X,Y)Z & = & \alpha {\rm R}(X,Y)Z+\beta[{\rm Ric}(Y,Z)X-{\rm Ric}(X,Z)Y] \\
& - & \frac{\tau}{n}\bigl(\frac{\alpha}{n-1} + \beta \bigl)[g(Y,Z)X-g(X,Z)Y] \nonumber
\end{eqnarray}
where ${\rm R}$ is curvature tensor, ${\rm Ric}$ is the Ricci tensor, $\tau$ is the scalar
curvature of $(M,g)$ and $\alpha,\beta \in \mathbb{R}-\{0\}.$

If $\alpha =1$ and $\beta =-\frac{1}{n-1}$ in $\eqref{pseudo projective curvature tensor}$, then
the pseudo-projective curvature tensor takes the following form
\begin{align}
\label{projective curvature tensor}
{\mathcal{P}}(X,Y)Z=R(X,Y)Z-\frac{1}{n-1}[S(Y,Z)X-S(X,Z)Y]
\end{align}
where $\mathcal{P}$ denotes the projective curvature tensor (see \cite{mishra}). Thus, the projective curvature tensor is a particular case of the pseudo-projective curvature tensor.  In a Riemannian manifold $M$, if there exists a one-to-one correspondence between each coordinate neighborhood of $M$ and a domain in the Euclidean space such that any geodesic of the Riemannian manifold corresponds to a straight line in the Euclidean space, then $M$ is said to be locally projectively flat. For $n \geq 1$, $M$ is locally projectively flat if and only if the projective curvature tensor $\mathcal{P}$ vanishes. Moreover, $\eqref{projective curvature tensor}$ implies that $M$ is projectively flat if and only if it is of constant curvature. Hence the projective curvature tensor is the measure of the failure of a Riemannian manifold to be of constant curvature. In \cite{mica}, projective curvature tensors of a non-symmetric affine connection space are expressed as functions of the affine connection coefficients and Weyl projective tensor of the corresponding associated affine connection space. Moreover, projective flatness of non-symmetric affine connection spaces is analyzed. There are also two recent studies (see \cite{Bhattacharyya21, Ayar21}) in the former one relations between pseudo-projectivity and warped products have been studied and in the latter one this structure was considered on Riemannian submersions equipped with quasi-conformal curvature. After that pseudo-projectively Ricci semi-symmetric and pseudo-projectively flat (or conservative) semi-Riemannian manifolds are studied in \cite{guler-f}. In this paper, our aim is to extend these study to the sequential warped product manifolds and to obtain their relativistic applications.

The warped product of two Riemannian manifolds $(B,g_B)$ and $(F,g_F)$ as the product $B\times F$ equipped with the metric  $g_B \oplus f^2 g_F$ and denoted by $B\times_f F$, where the smooth function $f \colon B \rightarrow (0,\infty)$ is called the warping function,  \cite{bishop}. Then warped products of semi-Riemannian (not necessarily Lorentzian) manifolds and their Riemannian and  Ricci curvature tensors were given in \cite{Oneill:1983}. Then it turned out that the standard spacetime models of the universe and many other fundamental examples of relativistic spacetimes as the solutions of Einstein's field equations can be expressed in terms of warped product manifolds. As a result of that, warped products play very important roles in differential geometry as well as in the theory of general relativity. Two very well-known  examples of this notion are four dimensional Schwarzschild and de Sitter solutions of Einstein field equations. Then, a new concept of warped product metric has been introduced in which its base or fiber or both  also have the warped product structure and it is called sequential warped product, \cite{shenawy,de filomat}. The formal definition of this notion is given as follows:

\begin{definition}
Let $(M_i,g_i)$ be three semi-Riemannian manifolds of dimensions $m_i,$ respectively
for any $i=1,2,3.$ Suppose that $f \colon M_1\rightarrow (0,\infty)$ and $h \colon
M_1\times M_2 \rightarrow (0,\infty)$ are two smooth functions. Then the sequential
warped product is the product manifold $\bar{M}=(M_1\times_f M_2)\times_h M_3$ endowed
with the metric
\begin{equation}
\label{sequential warped product metric}
\bar{g}=(g_1\oplus  f^2 g_2 ) \oplus h^2 g_3.
\end{equation}
Both  $f$ and $g$ are called the warping functions.
\end{definition}

\begin{notation} \label{notation1}

\begin{enumerate}
\item[(i)]  Throughout the paper, we use Einstein's summation convention.
All the objects will be assumed to be smooth and all the manifolds are connected.
\item[(ii)] All objects having ``bar" symbol represent the objects of the
sequential warped product manifold and all objects having the indices or powers
$i$ denote  the objects of the manifold $(M_i,g_i)$, where $i=1,2,3$.
\item[(iii)] The Riemann curvature tensor is defined by ${\rm R}(X,Y)Z=[\nabla_X, \nabla_Y]Z- \nabla_{[X,Y]}Z$,
the Ricci tensor is ${\rm Ric}(X,Y)=\sum_{i=1}^n {\rm R}(e_i, X, Y, e_i)$ and the scalar curvature
$\tau=\sum_{i=1}^n {\rm Ric}(e_i, e_i)$, where $\{e_i: i=1, \ldots , n\}$ denotes an orthonormal basis on the manifold.
\item[(iv)] For any $X, Y \in {\rm T}(M)$, the Hessian of a smooth function $\phi$ is the second order
covariant differentiation defined by ${\rm H}^{\phi}(X,Y)=XY(\phi)-(\nabla_XY)\phi=g(\nabla_X {\rm grad} \phi, Y)$.
\item[(v)] On a sequential warped product $\bar{M}=(M_1\times_f M_2)\times_h M_3$, every vector field $X$ can be decomposed as the sum
\begin{equation*}
\label{decomposition}
X=X_1+X_2+X_3, \ \ \textrm{where} \ \ X_i \in {\rm T}(M_i), \  \ i=1,2,3.
\end{equation*}
\end{enumerate}
\end{notation}

In the next section, we have provided all the necessary curvature formulas related to the geometry of
sequential warped products (for further details see \cite{guler-electronic,Karaca,Pahan}). In Section 3, we derive 
expressions for the pseudo-projective tensor formulas of a sequential warped product to obtain our main 
results. Then we consider pseudo-projectively flat sequential warped products and establish implications 
of this concept on the components of a sequential warped product. Finally, we apply the mains results to 
characterize pseudo-projectively flat sequential generalized Robertson-Walker and pseudo-projectively 
flat sequential standard static spacetimes.

\section{Preliminaries}

In this section, we give the basic formulas for the Levi-Civita connection,
Riemann, Ricci and the scalar curvature of the sequential warped products
that will be used throughout the study.

From now on, assume that $\bar{M}=(M_1\times_f M_2)\times_h M_3$ is a
sequential warped product furnished with the metric
$\bar{g}=(g_1\oplus  f^2 g_2 ) \oplus h^2 g_3$ and further assume that
$X_i,Y_i, Z_i$ are in ${\rm T}(M_i)$, for any $i=1,2,3$. Then, we have:

\begin{lemma} \cite{de filomat} \label{thm:connection}
The components of the Levi-Civita connection on $(\bar{M}, \bar{g})$ are given by:
\begin{enumerate}
\item $\bar{\nabla}_{X_1}Y_1=\nabla^1_{X_1}Y_1$,
\item $\bar{\nabla}_{X_1}X_2=\bar{\nabla}_{X_2}X_1=X_1(\ln f)X_2$,
\item $\bar{\nabla}_{X_2}Y_2=\nabla^2_{X_2}Y_2-fg_2(X_2,Y_2){\rm grad}^1f$,
\item $\bar{\nabla}_{X_3}X_1=\bar{\nabla}_{X_1}X_3=X_1(\ln h)X_3$,
\item $\bar{\nabla}_{X_2}X_3=\bar{\nabla}_{X_3}X_2=X_2(\ln h)X_3$,
\item $\bar{\nabla}_{X_3}Y_3=\nabla^3_{X_3}Y_3-hg_3(X_3,Y_3){\rm grad}h$.
\end{enumerate}
\end{lemma}

Note that, ${\rm H}_1^f$ and $\Delta_1 f$ denote the Hessian and Laplacian of
$f$ on $M_1$  respectively, while  ${\rm H}^h$ and $\Delta h$ denote the Hessian
and Laplacian of $h$ on $\bar{M}$, respectively.

\begin{lemma} \cite{de filomat} \label{thm:Riemannian curvature}
The non-zero components of the Riemannian curvature of $(\bar{M}, \bar{g})$ are given by:
\begin{enumerate}
\item $\bar{{\rm R}}(X_1,Y_1)Z_1={\rm R}_1(X_1,Y_1)Z_1$,
\item $\bar{{\rm R}}(X_2,Y_2)Z_2={\rm R}_2(X_2,Y_2)Z_2-||{\rm grad}_1f||^2[g_2(Y_2,Z_2)X_2-g_2(X_2,Z_2)Y_2]$,
\item $\bar{{\rm R}}(X_1,Y_2)Z_1=\frac{1}{f}{\rm H}_1^f(X_1,Z_1)Y_2$,
\item $\bar{{\rm R}}(X_1,Y_2)Z_2=-fg_2(Y_2,Z_2)\nabla^1_{X_1}{\rm grad}_1f $,
\item $\bar{{\rm R}}(X_i,Y_3)Z_j= \frac{1}{h}{\rm H}^h(X_i,Z_j)Y_3$, \ \ for $i,j=1,2$,
\item $\bar{{\rm R}}(X_i,Y_3)Z_3=-hg_3(Y_3,Z_3)\bar{\nabla}_{X_i}{\rm grad}h $,   $i=1,2$,
\item $\bar{{\rm R}}(X_3,Y_3)Z_3={\rm R}_3(X_3,Y_3)Z_3-||{\rm grad}h||^2[g_3(Y_3,Z_3)X_3-g_3(X_3,Z_3)Y_3]$.
\end{enumerate}
\end{lemma}

\begin{lemma} \cite{de filomat}\label{thm:Ricci curvature}
The non-zero components of the Ricci curvature of $(\bar{M}, \bar{g})$ are given by:
\begin{enumerate}
\item $ \bar{{\rm Ric}}(X_1,Y_1)= {\rm Ric}_1(X_1,Y_1)-\frac{m_2}{f}{\rm H}_1^f(X_1,Y_1)-\frac{m_3}{h}{\rm H}^h(X_1,Y_1)$,
\item $ \bar{{\rm Ric}}(X_2,Y_2)= {\rm Ric}_2(X_2,Y_2)-f^{\sharp}g_2(X_2,Y_2)-\frac{m_3}{h}{\rm H}^h(X_2,Y_2)$,
\item $ \bar{{\rm Ric}}(X_3,Y_3)= {\rm Ric}_3(X_3,Y_3)-h^{\sharp}g_3(X_3,Y_3) $,
\end{enumerate}
where $f^{\sharp}=f\Delta_1f+(m_2-1)||{\rm grad}_1f||^2$ and $h^{\sharp}=h\Delta h+(m_3-1)||{\rm grad}h||^2$.
\end{lemma}

\begin{lemma}\cite{de filomat} \label{thm:scalar curvature}
The relation between the scalar curvature $\bar{\tau}$ of $(\bar{M}, \bar{g})$ and the scalar curvatures
$\tau_i$ of $(M_i,g_i)$ for any $i=1,2,3,$ is given by:
\begin{eqnarray*}
\bar{\tau} & = & \tau_1+\frac{r_2}{f^2}+\frac{r_3}{h^2}-\frac{2m_2}{f}
\Delta_1f-\frac{2m_3}{h}\Delta h \\
& - & \frac{m_2(m_2-1)}{f^2}||{\rm grad}_1f||^2 - \frac{m_3(m_3-1)}{h^2}||{\rm grad}h||^2.
\end{eqnarray*}
\end{lemma}

In what follows, we use again Notations \ref{notation1} to obtain the following:

\begin{lemma}\label{thm:hessian} \cite{guler-electronic}
The Hessian tensor ${\rm H}^{\varphi}$ of $\varphi$ on  a sequential warped product $\bar{M}=(M_1\times_f M_2)\times_h M_3$  satisfies
\begin{align*}
(1) \ {\rm H}^{\varphi} (X_1,Y_1)&={\rm H}_1^{\varphi}(X_1,Y_1),\\ \notag
(2) \ {\rm H}^{\varphi} (X_1,Y_2)&=X_1(\ln f)Y_2(\varphi),\\ \notag
(3) \ {\rm H}^{\varphi} (X_1,Y_3)&=X_1(\ln h)Y_3(\varphi),\\ \notag
(4) \ {\rm H}^{\varphi} (X_2,Y_2)&=f\nabla \varphi(f)g_2(X_2,Y_2) +{\rm H}_2^{\varphi}(X_2,Y_2),\\ \notag
(5) \ {\rm H}^{\varphi} (X_2,Y_3)&=X_2(\ln h)Y_3(\varphi),\\ \notag
(6) \ {\rm H}^{\varphi} (X_3,Y_3)&=h\nabla \varphi(h)g_3(X_3,Y_3) +{\rm H}_3^{\varphi}(X_3,Y_3). \notag
\end{align*}

 \end{lemma}

 Also, note that for any function $\varphi$, the following relation holds:
 \begin{equation}
 \label{prem:eq.1}
 \displaystyle{\frac{m}{\varphi}{\rm H}^\varphi = {\rm H}^{m \ln \varphi} +
\frac{1}{m} {\rm d}(m \ln \varphi)
\otimes {\rm d}(m \ln \varphi)}, \ \ m \in \mathbb{R}.
 \end{equation}

\section{Pseudo-Projective Tensor on Sequential Warped Products }
In this section, some Ricci-Hessian class type manifolds are investigated on pseudo-projectively flat sequential warped products. If the pseudo-projective curvature tensor of $(M,g)$ vanishes identically, then it is called  pseudo-projectively flat.

\begin{remark} Recall that every pseudo-projectively manifold is also Einstein.
\end{remark}

An $n(>2$)-dimensional smooth manifold $(M^n,g)$ is said to be a  generalized quasi Einstein manifold in the sense of Catino  \cite{catino} 
if there exist three smooth functions $\varphi, \alpha$ and $\lambda$ on $M$ such that
\begin{equation}
\label{generalized quasi Einstein}
{\rm Ric}+{\rm H}^{\varphi}-\alpha {\rm d}\varphi \otimes {\rm d}\varphi=\lambda g
\end{equation}
and it is denoted by $(M^n,g,\varphi, \alpha, \lambda)$.  There are some different subclasses of this equation that define many important manifolds:
\begin{enumerate}
\item[(i)]If $\alpha =\frac{1}{m}$, for positive integer $0<m<\infty$ then $M$ is called an $m$-generalized quasi Einstein manifold.
\item[(ii)]If $\alpha =0$, then $M$ is called a gradient almost Ricci soliton, \cite{barros} and it is denoted by $(M,g, \varphi, \lambda)$ and $\varphi$ is called the potential function.
\item[(iii)] Another particular case of $\eqref{generalized quasi Einstein}$ can be written as
\begin{equation}
\label{h-grs}
{{\rm Ric}}+\psi{\rm H}^{\varphi}=\lambda {g}.
\end{equation}
This structure is said to be a $\psi$-almost gradient  Ricci soliton  \cite{yun} and  briefly denoted by $(M,g, \varphi, \psi, \lambda)$. For recent results for this class of solitons, we refer \cite{blaga-turki}. Note that, for every smooth function $\varphi$, the following relation can be verified:
\begin{equation}\label{h-4}
\begin{array}{ll}
{\rm H}^{\ln \varphi}=\displaystyle\frac{1}{\varphi}{\rm H}^{\varphi}-\displaystyle\frac{1}{\varphi^{2}}{\rm d}\varphi\otimes {\rm d}\varphi .
\end{array}
\end{equation}  By defining a function $\phi=e^{\frac{-\varphi}{m}}$, we get  $\frac{m}{\phi}{\rm H}^{\phi}=-{\rm H}^{\varphi}+\frac{1}{m} {\rm d}\varphi \otimes {\rm d}\varphi$. Thus  using $\eqref{h-4}$, $\eqref{generalized quasi Einstein}$ can be written as
\begin{equation}
{{\rm Ric}}-\frac{m}{\phi} {\rm H}^{\phi}=\lambda {g}.
\end{equation}
Hence,  any $m$-generalized quasi Einstein manifold is a $(\frac{-m}{\phi})$-almost gradient Ricci soliton.
\item[(iv)] More specificially, a  Riemannian manifold $(M^n, g)$ is called a conformal gradient soliton if there exists a non-constant smooth function $f$, called potential of the soliton, such that
\begin{equation}
\label{conformal gradient soliton}
\textrm{H}^f=\varphi g,
\end{equation} for some function $\varphi: M^n\rightarrow \mathbb{R}$.
 Tashiro  studied in \cite{tashiro}  complete Riemannian manifolds admitting a vector field $\nabla f$ satisfying equation $\eqref{conformal gradient soliton}$.
Then, in \cite{cheeger} Cheeger and Colding gave the solutions of the equation $\eqref{conformal gradient soliton}$ and they obtained the characterization of warped product manifolds. Also, the equation $\eqref{conformal gradient soliton}$ is the special case of gradient Ricci soliton structure, that has been studied extensively on the warped product manifold in \cite{pina}. For instance, in \cite{marco}  Marco and Gouthier gave the classification of the warped products of constant sectional curvature satisfying the condition $\eqref{conformal gradient soliton}$.
\end{enumerate}

\begin{theorem} \label{thm:projective curvature}
The non-zero components of the pseudo-projective curvature of a sequential warped product
$(\bar{M}, \bar{g})$ of the form $\bar{M}=(M_1\times_f M_2)\times_h M_3$ equipped with
$\bar{g}=(g_1\oplus  f^2 g_2 ) \oplus h^2 g_3$ are as follow:
\begin{enumerate}
\item \begin{eqnarray*} \bar{{\rm P}}(X_1,Y_1)Z_1 & = & \alpha {\rm R}_1(X_1,Y_1)Z_1 \\
& + & \beta[{\rm Ric}_1(Y_1,Z_1)X_1 - \frac{m_2}{f}{\rm H}_1^f(Y_1,Z_1)X_1 - \frac{m_3}{h}{\rm H}^h(Y_1,Z_1)X_1] \\
& - & \beta[{\rm Ric}_1(X_1,Z_1)Y_1 - \frac{m_2}{f}{\rm H}_1^f(X_1,Z_1)Y_1 - \frac{m_3}{h}{\rm H}^h(X_1,Z_1)Y_1] \\
& - & \frac{\tau}{n}\bigl( \frac{\alpha}{n-1}+\beta \bigl)[g_1(Y_1,Z_1)X_1-g_1(X_1,Z_1)Y_1],
\end{eqnarray*}
\item \begin{eqnarray*} \bar{{\rm P}}(X_2,Y_2)Z_2 & = & \alpha{\rm R}_2(X_2,Y_2)Z_2 - \alpha \|\nabla_1 f\|^2
[g_2(Y_2,Z_2)X_2-g_2(X_2,Z_2)Y_2] \\
& + & \beta[{\rm Ric}_2(Y_2,Z_2)X_2 - f^\sharp g_2(Y_2,Z_2)X_2 - \frac{m_3}{h}{\rm H}^h(Y_2,Z_2)X_2] \\
& - & \beta[{\rm Ric}_2(X_2,Z_2)Y_2 - f^\sharp g_2(X_2,Z_2)Y_2 - \frac{m_3}{h}{\rm H}^h(X_2,Z_2)Y_2] \\
& - & \frac{\tau}{n}\bigl( \frac{\alpha}{n-1}+\beta \bigl) f^2 [g_2(Y_2,Z_2)X_2-g_1(X_2,Z_2)Y_2],
\end{eqnarray*}
\item \begin{eqnarray*} \bar{{\rm P}}(X_1,Y_2)Z_1 & = & \frac{\alpha}{f}{\rm H}_1^f(X_1,Z_1)Y_2 +\frac{\tau}{n} \bigl( \frac{\alpha}{n-1}+\beta \bigl) g_1(X_1,Z_1)Y_2 \\
& - & \beta [{\rm Ric}_1(X_1,Z_1)Y_2 - \frac{m_2}{f}{\rm H}_1^f(X_1,Z_1)Y_2 - \frac{m_3}{h}{\rm H}^h(X_1,Z_1)Y_2 ],
\end{eqnarray*}
\item \begin{eqnarray*} \bar{{\rm P}}(X_1,Y_2)Z_2 & = & -\alpha f g_2(Y_2,Z_2) \nabla^1_{X_1} \nabla_1 f-\frac{\tau}{n} \bigl( \frac{\alpha}{n-1}+\beta \bigl) f^2 g_2(Y_2,Z_2)X_1 \\
& + & \beta [{\rm Ric}_2(Y_2,Z_2)X_1 - f^\sharp g_2(Y_2,Z_2)X_1 - \frac{m_3}{h}{\rm H}^h(Y_2,Z_2)X_1],
\end{eqnarray*}
\item \begin{eqnarray*} \bar{{\rm P}}(X_1,Y_3)Z_1 & = & \frac{\alpha}{f}{\rm H}_1^f(X_1,Z_1)Y_3 + \frac{\tau}{n} \bigl( \frac{\alpha}{n-1}+\beta \bigl) g_1(X_1,Z_1)Y_3 \\
& - & \beta [{\rm Ric}_1(X_1,Z_1)Y_3 - \frac{m_2}{f}{\rm H}_1^f(X_1,Z_1)Y_3 - \frac{m_3}{h}{\rm H}^h(X_1,Z_1)Y_3 ],
\end{eqnarray*}
\item $\displaystyle{\bar{{\rm P}}(X_1,Y_3)Z_2 = \frac{\alpha}{h} {\rm H}^h(X_1,Z_2)Y_3 },$
\item $\displaystyle{\bar{{\rm P}}(X_2,Y_3)Z_1 = \frac{\alpha}{h} {\rm H}^h(X_2,Z_1)Y_3 },$
\item \begin{eqnarray*} \bar{{\rm P}}(X_2,Y_3)Z_2 & = & \frac{\alpha}{h}{\rm H}_1^f(X_2,Z_2)Y_3 +\frac{\tau}{n} \bigl( \frac{\alpha}{n-1}+\beta \bigl) f^2 g_2(X_2,Z_2)Y_3\\
& - & \beta [{\rm Ric}_2(X_2,Z_2)Y_3 - f^\sharp g_2(X_2,Z_2)Y_3 - \frac{m_3}{h}{\rm H}^h(X_2,Z_2)Y_3 ],
\end{eqnarray*}
\item \begin{eqnarray*} \bar{{\rm P}}(X_1,Y_3)Z_3 & = & -\alpha h g_3(Y_3,Z_3) \nabla_{X_1} \nabla h  -\frac{\tau}{n} \bigl( \frac{\alpha}{n-1}+\beta \bigl) h^2 g_3(Y_3,Z_3)X_1\\
& + & \beta [{\rm Ric}_3(Y_3,Z_3)X_1 - h^\sharp g_3(Y_3,Z_3)X_1],
\end{eqnarray*}
\item \begin{eqnarray*} \bar{{\rm P}}(X_2,Y_3)Z_3 & = & -\alpha h g_3(Y_3,Z_3) \nabla_{X_2} \nabla h -\frac{\tau}{n} \bigl( \frac{\alpha}{n-1}+\beta \bigl) h^2 g_3(Y_3,Z_3)X_2 \\
& + & \beta [{\rm Ric}_3(Y_3,Z_3)X_2 - h^\sharp g_3(Y_3,Z_3)X_2],
\end{eqnarray*}
\item \begin{eqnarray*} \bar{{\rm P}}(X_3,Y_3)Z_3 & = & \alpha{\rm R}_3(X_3,Y_3)Z_3 -\frac{\tau}{n}\bigl( \frac{\alpha}{n-1}+\beta \bigl) h^2 [g_3(Y_3,Z_3)X_3-g_3(X_3,Z_3)Y_3]\\
& + & \beta[{\rm Ric}_3(Y_3,Z_3)X_3 - h^\sharp g_3(Y_3,Z_3)X_3] \\
& - & \beta[{\rm Ric}_3(X_3,Z_3)Y_3 - h^\sharp g_3(X_3,Z_3)Y_3],
\end{eqnarray*}
\end{enumerate}
where $X_i \in {\rm T}(M_i)$ for any $i=1,2,3$ and $n=m_1+m_2+m_3.$
\end{theorem}


Now, we assume that the  sequential warped product
$(\bar{M}, \bar{g})$ of the form $\bar{M}=(M_1\times_f M_2)\times_h M_3$ equipped with
$\bar{g}=(g_1\oplus  f^2 g_2 ) \oplus h^2 g_3$ is pseudo-projectively flat, i.e. $\bar{P}=0$ on $(\bar{M}, \bar{g})$.

After setting equations (1) and (11) of Theorem \ref{thm:projective curvature},
contract the resulting equation over $Y_1$ and $Z_1,$ we obtain that:

\begin{multline} \label{EqnI}
(\alpha-\beta){\rm Ric}_1(X_1,W_1) +
\beta m_2 \frac{1}{f} {\rm H}_1^f(X_1,W_1) + \beta m_3 \frac{1}{h}{\rm H}^h(X_1,W_1) \\
=\Bigl[\frac{\tau}{n}\bigl(\frac{\alpha}{n-1}+\beta \bigl)(m_1-1)-\beta \tau_1 + \beta m_2 \frac{1}{f}
\Delta_1 f + \beta m_3 \frac{1}{h} \Delta_1 h \Bigl]g_1(X_1,W_1).
\end{multline}
Since $\bar{{\rm P}}=0,$ by (3) of Theorem $\ref{thm:projective curvature}$ we have:
\begin{multline} \label{EqnII}
\frac{\alpha}{f}{\rm H}_1^f(X_1,Z_1) - \beta{\rm Ric}_1(X_1,Z_1)+
\beta \frac{m_2}{f}{\rm H}_1^f(X_1,Z_1)+\beta \frac{m_3}{h}{\rm
H}^h(X_1,Z_1) \\
+\frac{\tau}{n} \bigl(\frac{\alpha}{n-1}+\beta \bigl) g_1(X_1,Z_1)=0.
\end{multline}
By combining equations \eqref{EqnI} and \eqref{EqnII} also noting that
$\alpha \neq 0$ we obtain that
\begin{equation}
{\rm Ric}_1 - \frac{1}{f} {\rm H}^f_1 = \frac{\lambda_1}{\alpha}g_1,
\end{equation}
where $$\lambda_1=\frac{\tau}{n} \bigl(\frac{\alpha}{n-1}+\beta
\bigl)m_1-\beta \tau_1 + \beta \frac{m_2}{f} \Delta_1f + \beta
\frac{m_3}{h} \Delta_1h. $$
Thus, by $\eqref{prem:eq.1}$, the above equation can be expressed as
$$ {\rm Ric}_1 - {\rm H}^{\ln f}_1 - {\rm d}( \ln f)
\otimes {\rm d}(\ln f) = \frac{\lambda_1}{\alpha}g_1. $$
More explicitly, $(M_1,g_1,-\ln f, \lambda_1/\alpha)$ is a
generalized quasi-Einstein manifold in the sense of Catino, where
the potential function is determined in terms of the warping
function.
\\
Again since $\bar{{\rm P}}=0,$ the item (2) of Theorem \ref{thm:projective curvature}
becomes:
\begin{eqnarray*} - \alpha{\rm R}_2(X_2,Y_2)Z_2 & = & - \alpha \|\nabla_1 f\|^2
[g_2(Y_2,Z_2)X_2-g_2(X_2,Z_2)Y_2] \\
& + & \beta[{\rm Ric}_2(Y_2,Z_2)X_2 - f^\sharp g_2(Y_2,Z_2)X_2 - \frac{m_3}{h}{\rm H}^h(Y_2,Z_2)X_2] \\
& - & \beta[{\rm Ric}_2(X_2,Z_2)Y_2 - f^\sharp g_2(X_2,Z_2)Y_2 - \frac{m_3}{h}{\rm H}^h(X_2,Z_2)Y_2] \\
& - & \frac{\tau}{n}\bigl( \frac{\alpha}{n-1}+\beta \bigl) f^2 [g_2(Y_2,Z_2)X_2-g_1(X_2,Z_2)Y_2].
\end{eqnarray*}
Contracting the last equation over $Y_2$ and $Z_2$ and then using the identity (4) of Lemma \ref{thm:hessian}, we obtain:
\begin{equation} \label{2P}
(\alpha-\beta){\rm Ric}_2(X_2,W_2)+ \frac{\mu}{h} {\rm H}^h_2(X_2,W_2) = \lambda g_2(X_2,W_2),
\end{equation}
where \begin{multline*}
\lambda_2 = -\beta \tau_2 + \beta m_2 f^\sharp + \beta m_3 (m_2 f \nabla h(f)+f^2\Delta h) \frac{1}{h} -
\beta f^\sharp \\
- \beta m_3 \frac{f}{h} \nabla h(f) + \alpha \|\nabla_1 f\|^2 +
\frac{\tau}{n} \bigl( \frac{\alpha}{n-1} + \beta \bigl) f^2
\end{multline*}

and $\displaystyle{\mu (\alpha-\beta)=\beta m_3 f^2}.$

If $\alpha \neq \beta,$ then by  \eqref{2P} and \eqref{prem:eq.1}, we can state
$${\rm Ric}_2 + {\rm H}^{\mu \ln h} + \frac{1}{\mu} {\rm d}(\mu \ln h) \otimes {\rm d}(\mu \ln h)
= \lambda_2 g_2.$$
More explicitly, $(M_2,g_2,\mu \ln h, \lambda_2)$ is a
generalized quasi-Einstein manifold in the sense of Catino, where
the potential function is determined in terms of the warping
function. If $\alpha = \beta,$ then $\eqref{2P}$ reduces to the relation of conformal gradient soliton.

Notice that (6) and (7) of Theorem \ref{thm:projective curvature} imply that
${\rm H}^h(X_1,Z_2) = 0$ and ${\rm H}^h(X_2,Z_1)=0.$
\\
Lastly, since $\bar{{\rm P}}=0,$ the item (11) of Theorem \ref{thm:projective curvature}
implies:

\begin{eqnarray*} -\alpha{\rm R}_3(X_3,Y_3)Z_3 & = & \beta[{\rm Ric}_3(Y_3,Z_3)X_3 - h^\sharp g_3(Y_3,Z_3)X_3] \\
& - & \beta[{\rm Ric}_3(X_3,Z_3)Y_3 - h^\sharp g_3(X_3,Z_3)Y_3] \\
& - & \frac{\tau}{n}\bigl( \frac{\alpha}{n-1}+\beta \bigl) h^2 [g_3(Y_3,Z_3)X_3-g_3(X_3,Z_3)Y_3].
\end{eqnarray*}
After contracting the last equation over $Y_3$ and $Z_3$, we have
\begin{eqnarray*} (\alpha - \beta) {\rm Ric}_3  = [\frac{\tau}{n} \bigl( \frac{\alpha}{n-1}+\beta \bigl)
h^2(m_3-1)-  \beta \tau_3 + \beta h^\sharp (m_3-1)] g_3.
\end{eqnarray*}
That is, $(M_3,g_3)$ is Einstein.

All the above findings can be summarized as:

\begin{theorem} \label{thm:pseudo projective flat}
If  a sequential warped product
$(\bar{M}, \bar{g})$ of the form $\bar{M}=(M_1\times_f M_2)\times_h M_3$ equipped with
$\bar{g}=(g_1\oplus  f^2 g_2 ) \oplus h^2 g_3$ is pseudo-projectively flat, then
\begin{enumerate}
\item $(M_1,g_1,-\ln f, \lambda_1/\alpha)$ is a
generalized quasi-Einstein manifold,
\item $(M_2,g_2,\mu \ln h, \lambda_2)$ is a
generalized quasi-Einstein manifold, provided that $\alpha \neq \beta$. Otherwise, it is conformal gradient soliton.
\item $(M_3,g_3)$ is Einstein,  provided that $\alpha \neq \beta$ and
\item $\displaystyle{\frac{1}{f}{\rm H}_1^f = \frac{1}{h}{\rm H}_1^h}$ on $M_1.$
\end{enumerate}
\end{theorem}

\begin{remark}
If $\alpha=1$ and $\beta=-1/(n-1)$, then $\bar{{\rm P}}$ is reduced to the
projective curvature tensor. Hence, Theorem \ref{thm:pseudo projective flat} holds also for
projectively flat sequential warped products.
\end{remark}
Note that if the Ricci tensor of the manifold satisfies the condition
\begin{align}
\label{codazzi}
(\nabla_X {\rm Ric})(Y,Z)=(\nabla_Y {\rm Ric})(X,Z),
\end{align}
which means that the manifold has the Codazzi type Ricci tensor, \cite{Besse2008,ferus}.  Hence,  we can state the following:

\begin{corollary} Let  $(\bar{M}, \bar{g})$ be a sequential warped product
of the form $\bar{M}=(M_1\times_f M_2)\times_h M_3.$ Assume that $(\bar{M}, \bar{g})$
is pseudo-projectively conservative sequential warped product. If
$\alpha+\beta \neq 0,$ and $(\bar{M}, \bar{g})$ is of constant scalar
curvature then the Ricci tensor of $(\bar{M}, \bar{g})$ is of Codazzi type.
\end{corollary}

\begin{proof} Note that if ${\rm div}(\bar{{\rm P}})=0,$ then ${\rm div}\bar{{\rm C}}=0$ or
$\alpha+\beta \neq 0,$ where $\bar{{\rm C}}$ denotes for Weyl tensor.
\end{proof}

\section{Applications}

\subsection{Sequential Generalized Robertson-Walker Spacetimes}

A sequential generalized Robertson-Walker spacetime is a sequential warped product
of the form $\bar{M}=(I \times_f M_2)\times_h M_3$ endowed with the metric
$\bar{g}=(-\textrm{d}t^2\oplus  f^2 g_2 ) \oplus h^2 g_3.$

By using Thorem \ref{thm:projective curvature}, one can easily state the following:

\begin{theorem} \label{thm:projective curvature SGRW}
The non-zero components of the pseudo-projective curvature of a sequential generalized Robertson-Walker
spacetime $(\bar{M}, \bar{g})$ of the form $\bar{M}=(I \times_f M_2)\times_h M_3$ endowed with the metric
$\bar{g}=(-\textrm{d}t^2\oplus  f^2 g_2 ) \oplus h^2 g_3$ are as follow:
\begin{enumerate}
\item \begin{eqnarray*} \bar{{\rm P}}(X_2,Y_2)Z_2 & = & \alpha{\rm R}_2(X_2,Y_2)Z_2 + \alpha (\dot{f})^2
[g_2(Y_2,Z_2)X_2-g_2(X_2,Z_2)Y_2] \\
& + & \beta[{\rm Ric}_2(Y_2,Z_2)X_2 - f^\sharp g_2(Y_2,Z_2)X_2 - \frac{m_3}{h}{\rm H}^h(Y_2,Z_2)X_2] \\
& - & \beta[{\rm Ric}_2(X_2,Z_2)Y_2 - f^\sharp g_2(X_2,Z_2)Y_2 - \frac{m_3}{h}{\rm H}^h(X_2,Z_2)Y_2] \\
& - & \frac{\tau}{n}\bigl( \frac{\alpha}{n-1}+\beta \bigl) f^2 [g_2(Y_2,Z_2)X_2-g_1(X_2,Z_2)Y_2],
\end{eqnarray*}
\item \begin{eqnarray*} \bar{{\rm P}}(\partial_t,Y_2)\partial_t & = & - \alpha \frac{\ddot{f}}{f}Y_2
- \beta \bigl[m_2 \frac{\ddot{f}}{f} + m_3 \frac{1}{h} \frac{\partial^2 h}{\partial t^2} \bigl] Y_2  -  \frac{\tau}{n} \bigl( \frac{\alpha}{n-1}+\beta \bigl) Y_2,
\end{eqnarray*}
\item \begin{eqnarray*} \bar{{\rm P}}(\partial_t,Y_2)Z_2 & = & -\alpha f \ddot{f} g_2(Y_2,Z_2) \partial_t -\frac{\tau}{n} \bigl( \frac{\alpha}{n-1}+\beta \bigl) f^2 g_2(Y_2,Z_2) \partial_t\\
& + & \beta [{\rm Ric}_2(Y_2,Z_2)- f^\sharp g_2(Y_2,Z_2) - \frac{m_3}{h}{\rm H}^h(Y_2,Z_2)] \partial_t,
\end{eqnarray*}
\item \begin{eqnarray*} \bar{{\rm P}}(\partial_t,Y_3)\partial_t & = & - \frac{\alpha}{h} \frac{\partial^2 h}{\partial t^2} Y_3  -  \beta \bigl[m_2 \frac{\ddot{f}}{f} - m_3 \frac{1}{h} \frac{\partial^2 h}{\partial t^2} \bigl] Y_3 -  \frac{\tau}{n} \bigl( \frac{\alpha}{n-1}+\beta \bigl) Y_3,
\end{eqnarray*}
\item \begin{eqnarray*} \bar{{\rm P}}(X_2,Y_3)Z_2 & = & \frac{\alpha}{h}{\rm H}_1^f(X_2,Z_2)Y_3 + \frac{\tau}{n} \bigl( \frac{\alpha}{n-1}+\beta \bigl) f^2 g_2(X_2,Z_2)Y_3 \\
& - & \beta [{\rm Ric}_2(X_2,Z_2)Y_3 - f^\sharp g_2(X_2,Z_2)Y_3 - \frac{m_3}{h}{\rm H}^h(X_2,Z_2)Y_3 ],
\end{eqnarray*}
\item \begin{eqnarray*} \bar{{\rm P}}(\partial_t,Y_3)Z_3 & = & -\alpha h g_3(Y_3,Z_3) \nabla_{\partial_t} \nabla h - \frac{\tau}{n} \bigl( \frac{\alpha}{n-1}+\beta \bigl) h^2 g_3(Y_3,Z_3) \partial_t \\
& + & \beta [{\rm Ric}_3(Y_3,Z_3)- h^\sharp g_3(Y_3,Z_3)] \partial_t,
\end{eqnarray*}
\item \begin{eqnarray*} \bar{{\rm P}}(X_2,Y_3)Z_3 & = & -\alpha h g_3(Y_3,Z_3) \nabla_{X_2} \nabla h -\frac{\tau}{n} \bigl( \frac{\alpha}{n-1}+\beta \bigl) h^2 g_3(Y_3,Z_3)X_2\\
& + & \beta [{\rm Ric}_3(Y_3,Z_3)X_2 - h^\sharp g_3(Y_3,Z_3)X_2],
\end{eqnarray*}
\item \begin{eqnarray*} \bar{{\rm P}}(X_3,Y_3)Z_3 & = & \alpha{\rm R}_3(X_3,Y_3)Z_3-\frac{\tau}{n}\bigl( \frac{\alpha}{n-1}+\beta \bigl) h^2 [g_3(Y_3,Z_3)X_3-g_3(X_3,Z_3)Y_3] \\
& + & \beta[{\rm Ric}_3(Y_3,Z_3)X_3 - h^\sharp g_3(Y_3,Z_3)X_3] \\
& - & \beta[{\rm Ric}_3(X_3,Z_3)Y_3 - h^\sharp g_3(X_3,Z_3)Y_3],
\end{eqnarray*}
\end{enumerate}
where $X_i \in {\rm T}(M_i)$ for any $i=1,2,3$ and $n=1+m_2+m_3.$
\end{theorem}
Now, we will consider implications of the case given by $\bar{{\rm P}}=0.$
\\
Equations (2) and (4) of Theorem \ref{thm:projective curvature SGRW} take
the following forms, respectively:
\begin{eqnarray*}
- \alpha \frac{\ddot{f}}{f}
- \beta \bigl[m_2 \frac{\ddot{f}}{f} + m_3 \frac{1}{h} \frac{\partial^2 h}{\partial t^2}]
- \frac{\tau}{n} \bigl( \frac{\alpha}{n-1}+\beta \bigl) & = & 0, \\
-\frac{\alpha}{h} \frac{\partial^2 h}{\partial t^2} -
\beta \bigl[m_2 \frac{\ddot{f}}{f} + m_3 \frac{1}{h} \frac{\partial^2 h}{\partial t^2}]
- \frac{\tau}{n} \bigl( \frac{\alpha}{n-1}+\beta \bigl) & = & 0.
\end{eqnarray*}
In view of the last two equations, one can easily obtain: $$\frac{\ddot{f}}{f} = \frac{1}{h} \frac{\partial^2 h}{\partial t^2}.$$

Similarly, equations (3) and (5) of Theorem \ref{thm:projective curvature SGRW} become
the followings, respectively:
\begin{eqnarray*}
-\alpha f \ddot{f} g_2 + \beta [{\rm Ric}_2 - f^\sharp g_2 - \frac{m_3}{h}{\rm H}^h]
- \frac{\tau}{n} \bigl( \frac{\alpha}{n-1}+\beta \bigl) f^2 g_2 & = & 0, \\
\frac{\alpha}{h}{\rm H}_1^f - \beta [{\rm Ric}_2 - f^\sharp g_2 - \frac{m_3}{h}{\rm H}^h]
 +  \frac{\tau}{n} \bigl( \frac{\alpha}{n-1}+\beta \bigl) f^2 g & = & 0.
\end{eqnarray*}
Hence, from  the last two equations $$ \frac{1}{h} {\rm H}^h = f \ddot{f} g_2 \quad \text{on} \quad M_2. \quad $$
The last equation imply that
$${\rm H}_2^h = \bigl[ h \frac{\ddot{f}}{f} - \frac{1}{f} \nabla h(f) \bigl]g_2.$$
That is, $(M_2,g_2)$ is a conformal gradient soliton.
Also, by contracting the previous equation, we get:
$$\Delta_2 h = m_2 \bigl[  h \frac{\ddot{f}}{f} - \frac{1}{f} \nabla h(f) \bigl].$$

Moreover, from equations (6) and (7) of Theorem \ref{thm:projective curvature SGRW}, the warping function $h \in \mathcal C^\infty(M_1 \times M_2)$ satisfies
we have:
$$ \nabla_{\partial t} \nabla h = \nabla_{X_2} \nabla h. $$
\\
Again setting $\bar{{\rm P}}=0$ in equation (3) of Theorem \ref{thm:projective curvature SGRW}
and then combining the resulting equation with (1) of Theorem \ref{thm:projective curvature SGRW}
(note that $\alpha \neq 0$) result in:
$${\rm R}_2(X_2,Y_2)Z_2 = -\bigl( f\ddot{f}+(\dot{f})^2 \bigl)
[g_2(Y_2,Z_2)X_2 - g_2(X_2,Z_2)Y_2]. $$

Hence, $(M_2,g_2)$ is of constant sectional curvature with
$k_2=-\bigl( f\ddot{f}+(\dot{f})^2 \bigl)$ and thus

$${\rm Ric}_2(X_2,Y_2)=-\bigl( f\ddot{f}+(\dot{f})^2 \bigl)(m_2-1)g_2(X_2,Y_2).$$
That is, $(M_2,g_2)$ is Einstein. Similarly, $(M_3,g_3)$ is Einstein as well.
\\
Therefore, we can state the following result.

\begin{theorem} \label{thm:flat projective curvature SGRW}
Let $(\bar{M}, \bar{g})$ be a sequential generalized Robertson-Walker spacetime
of the form $\bar{M}=(I \times_f M_2)\times_h M_3$ endowed with the metric
$\bar{g}=(-\textrm{d}t^2\oplus  f^2 g_2 ) \oplus h^2 g_3.$ If $(\bar{M}, \bar{g})$
is pseudo-projectively flat, then
\begin{enumerate}
\item the warping functions $f$ and $h$ satisfy
$$\frac{\ddot{f}}{f} = \frac{1}{h} \frac{\partial^2 h}{\partial t^2},$$
\item $(M_2,g_2)$ is a conformal gradient soliton and both $(M_2,g_2)$ and $(M_3,g_3)$ are Einstein.
\end{enumerate}
\end{theorem}

If $f\ddot{f}+(\dot{f})^2 =0$, then $(f\dot{f})=$constant, which yields $f^2(t)=at +b$, for some constant $a, b$.  Thus, we can state the following:

\begin{corollary}
Let $(\bar{M}, \bar{g})$ be a sequential generalized Robertson-Walker spacetime
of the form $\bar{M}=(I \times_f M_2)\times_h M_3$ endowed with the metric
$\bar{g}=(-\textrm{d}t^2\oplus  f^2 g_2 ) \oplus h^2 g_3.$ If $(\bar{M}, \bar{g})$
is pseudo-projectively flat and $f^2(t)=at +b$, for some constant $a, b$,   then $(M_2,g_2)$ is flat.
\end{corollary}

\subsection{Sequential Standard Static Spacetimes}

A sequential standard static spacetime is a sequential warped product
of the form $\bar{M}=(M_1\times_f M_2)\times_h I$ endowed with the metric
$\bar{g}=(g_1\oplus  f^2 g_2 ) \oplus h^2 (-\textrm{d}t^2)$.
As an application of Theorem \ref{thm:projective curvature} and Theorem \ref{thm:pseudo projective flat}, the following results are directly obtained. Therefore, we omit their proofs.
\begin{theorem} \label{thm:projective curvature SSST}
The non-zero components of the pseudo-projective curvature of a sequential standard
static spacetime $(\bar{M}, \bar{g})$ of the form $\bar{M}=(M_1\times_f M_2)\times_h I$ equipped with
$\bar{g}=(g_1\oplus  f^2 g_2 ) \oplus h^2 (-\textrm{d}t^2)$ are as follow:
\begin{enumerate}
\item \begin{eqnarray*} \bar{{\rm P}}(X_1,Y_1)Z_1 & = & \alpha {\rm R}_1(X_1,Y_1)Z_1 \\
& + & \beta[{\rm Ric}_1(Y_1,Z_1)X_1 - \frac{m_2}{f}{\rm H}_1^f(Y_1,Z_1)X_1 - \frac{m_3}{h}{\rm H}^h(Y_1,Z_1)X_1] \\
& - & \beta[{\rm Ric}_1(X_1,Z_1)Y_1 - \frac{m_2}{f}{\rm H}_1^f(X_1,Z_1)Y_1 - \frac{m_3}{h}{\rm H}^h(X_1,Z_1)Y_1] \\
& - & \frac{\tau}{n}\bigl( \frac{\alpha}{n-1}+\beta \bigl)[g_1(Y_1,Z_1)X_1-g_1(X_1,Z_1)Y_1],
\end{eqnarray*}
\item \begin{eqnarray*} \bar{{\rm P}}(X_2,Y_2)Z_2 & = & \alpha{\rm R}_2(X_2,Y_2)Z_2 - \alpha \|\nabla_1 f\|^2
[g_2(Y_2,Z_2)X_2-g_2(X_2,Z_2)Y_2] \\
& + & \beta[{\rm Ric}_2(Y_2,Z_2)X_2 - f^\sharp g_2(Y_2,Z_2)X_2 - \frac{m_3}{h}{\rm H}^h(Y_2,Z_2)X_2] \\
& - & \beta[{\rm Ric}_2(X_2,Z_2)Y_2 - f^\sharp g_2(X_2,Z_2)Y_2 - \frac{m_3}{h}{\rm H}^h(X_2,Z_2)Y_2] \\
& - & \frac{\tau}{n}\bigl( \frac{\alpha}{n-1}+\beta \bigl) f^2 [g_2(Y_2,Z_2)X_2-g_1(X_2,Z_2)Y_2],
\end{eqnarray*}
\item \begin{eqnarray*} \bar{{\rm P}}(X_1,Y_2)Z_1 & = & \frac{\alpha}{f}{\rm H}_1^f(X_1,Z_1)Y_2 +\frac{\tau}{n} \bigl( \frac{\alpha}{n-1}+\beta \bigl) g_1(X_1,Z_1)Y_2\\
& - & \beta [{\rm Ric}_1(X_1,Z_1)Y_2 - \frac{m_2}{f}{\rm H}_1^f(X_1,Z_1)Y_2 - \frac{m_3}{h}{\rm H}^h(X_1,Z_1)Y_2 ],
\end{eqnarray*}
\item \begin{eqnarray*} \bar{{\rm P}}(X_1,Y_2)Z_2 & = & -\alpha f g_2(Y_2,Z_2) \nabla^1_{X_1} \nabla_1 f-\frac{\tau}{n} \bigl( \frac{\alpha}{n-1}+\beta \bigl) f^2 g_2(Y_2,Z_2)X_1 \\
& + & \beta [{\rm Ric}_2(Y_2,Z_2)X_1 - f^\sharp g_2(Y_2,Z_2)X_1 - \frac{m_3}{h}{\rm H}^h(Y_2,Z_2)X_1],
\end{eqnarray*}
\item \begin{eqnarray*} \bar{{\rm P}}(X_1,\partial_t)Z_1 & = & \frac{\alpha}{f}{\rm H}_1^f(X_1,Z_1)\partial_t +\frac{\tau}{n} \bigl( \frac{\alpha}{n-1}+\beta \bigl) g_1(X_1,Z_1)\partial_t \\
& - & \beta [{\rm Ric}_1(X_1,Z_1)\partial_t - \frac{m_2}{f}{\rm H}_1^f(X_1,Z_1)\partial_t - \frac{m_3}{h}{\rm H}^h(X_1,Z_1)\partial_t ],
\end{eqnarray*}
\item $\displaystyle{\bar{{\rm P}}(X_1,\partial_t)Z_2 = \frac{\alpha}{h} {\rm H}^h(X_1,Z_2)\partial_t },$
\item $\displaystyle{\bar{{\rm P}}(X_2,\partial_t)Z_1 = \frac{\alpha}{h} {\rm H}^h(X_2,Z_1)\partial_t },$
\item \begin{eqnarray*} \bar{{\rm P}}(X_2,\partial_t)Z_2 & = & \frac{\alpha}{h}{\rm H}_1^f(X_2,Z_2)\partial_t + \frac{\tau}{n} \bigl( \frac{\alpha}{n-1}+\beta \bigl) f^2 g_2(X_2,Z_2)\partial_t\\
& - & \beta [{\rm Ric}_2(X_2,Z_2)\partial_t - f^\sharp g_2(X_2,Z_2)\partial_t - \frac{m_3}{h}{\rm H}^h(X_2,Z_2)\partial_t ],
\end{eqnarray*}
\item \begin{eqnarray*} \bar{{\rm P}}(X_1,\partial_t)\partial_t & = & \alpha h \nabla_{X_1} \nabla h +  \beta h^\sharp X_1 + \frac{\tau}{n} \bigl( \frac{\alpha}{n-1}+\beta \bigl) h^2 X_1,
\end{eqnarray*}
\item \begin{eqnarray*} \bar{{\rm P}}(X_2,\partial_t)\partial_t & = & \alpha h \nabla_{X_2} \nabla h +  \beta h^\sharp X_2 + \frac{\tau}{n} \bigl( \frac{\alpha}{n-1}+\beta \bigl) h^2 X_2,
\end{eqnarray*}
\end{enumerate}
where $X_i \in {\rm T}(M_i)$ for any $i=1,2,3$ and $n=1+m_2+m_3.$
\end{theorem}

\begin{theorem} \label{thm:flat projective curvature SSST}
Let $(\bar{M}, \bar{g})$ be a sequential standard static spacetime
of the form $\bar{M}=(M_1\times_f M_2)\times_h I$ endowed with the metric
$\bar{g}=(g_1\oplus  f^2 g_2 ) \oplus h^2 (-\textrm{d}t^2).$ If $(\bar{M}, \bar{g})$
is pseudo-projectively flat, then
\begin{enumerate}
\item $(M_1,g_1,-\ln f, \lambda_1/\alpha)$ is a
generalized quasi-Einstein manifold,
\item $(M_2,g_2,\mu \ln h, \lambda_2)$ is a
generalized quasi-Einstein manifold, and
\item $\displaystyle{\frac{1}{f}{\rm H}_1^f = \frac{1}{h}{\rm H}_1^h}$ on $M_1.$
\end{enumerate}
\end{theorem}

\end{document}